\documentclass[a4paper]{amsart}
\pdfoutput=1

\usepackage[utf8]{inputenc}
\usepackage[T1]{fontenc}
\usepackage{lmodern}
\usepackage{microtype}

\usepackage[french,english]{babel}
\usepackage{csquotes}

\usepackage{amsmath,amsthm,amsfonts,amssymb}
\usepackage{mathtools}

\usepackage[shortlabels]{enumitem}
\usepackage[pdftex,hidelinks,hypertexnames=false]{hyperref}

\usepackage[
  citestyle=authoryear-comp,
  bibstyle=authoryear,
  backend=bibtex,
  isbn=false,
  url=false,
  doi=true,
  maxcitenames=3,
  maxbibnames=6,
  firstinits=true
]{biblatex}
\DeclareNameAlias{sortname}{first-last}
\bibliography{biblio}

\newcommand\mop[1]{\operatorname{#1}}

\def\bC{\mathbb{C}}
\def\bN{\mathbb{N}}

\def\bP{\mathbb{P}}

\def\leq{\leqslant}
\def\geq{\geqslant}

\def\cI{\mathcal{I}}

\newcommand{\sst}{{\mathrm{ss}}}
\newcommand{\dquo}{/\!/}

\DeclareMathOperator{\GL}{GL}
\DeclareMathOperator{\SL}{SL}
\DeclareMathOperator{\End}{End}
\newcommand{\Bsk}{B_\mathrm{skew}}

\def\detdrei{\mathrm{det}_3}

\def\eqdef{\stackrel{\text{def}}{=}}
\newcommand{\st}{\mathrel{}\middle|\mathrel{}}

\newtheorem{theorem}{Theorem}

\newtheorem{lemma}[theorem]{Lemma}

\theoremstyle{definition}

\theoremstyle{remark}

\newcommand{\newabstract}[1]{%
  \par\bigskip
  \csname otherlanguage*\endcsname{#1}%
  \csname captions#1\endcsname
  \item[\hskip\labelsep\scshape\abstractname.]
}

\title[The $3\times 3$ determinant polynomial]%
{The boundary of the orbit of\\the 3 by 3 determinant polynomial\\
  \rule[\smallskipamount]{2cm}{.5pt}\\
La frontière de l'orbite\\
du polynôme déterminant 3 par 3.
}
\author[J.~Hüttenhain]{Jesko Hüttenhain}
\author[P.~Lairez]{Pierre Lairez}
\address{
  Technische Universität Berlin\\
  Straße des 17. Juni 136\\
  10623 Berlin\\
  Germany
}
\thanks{Partially funded by the research grant BU~1371/2-2 of the Deutsche Forschungsgemeinschaft.}
\email[J.~Hüttenhain]{ jesko@math.tu-berlin.de }
\email[P.~Lairez]{ pierre@lairez.fr }
\subjclass[2010]{14L35 (14E05)}

\begin{document}

\begin{abstract}
  We consider the $3\times 3$ determinant polynomial and we describe the limit
  points of the set of all polynomials obtained from the determinant polynomial
  by linear change of variables.  This answers a question of J.~M.~Landsberg.

  \newabstract{french}
  Nous étudions le polynôme donné par le déterminant~$3\times 3$ et décrivons l'adhérence
  de l'ensemble des polynômes obtenus par changements de variables linéaires à partir de ce déterminant,
  ce qui répond à une question de J.~M.~Lansberg.
\end{abstract}

\maketitle

\section*{Introduction}

\Textcite{MulSoh01} propose, in their \emph{geometric complexity theory}, to
study the geometry of the orbit closure of some polynomials under linear change
of variables, and especially, the determinant polynomial.  Yet, very few
explicit results describing the geometry are known in low dimension.  The
purpose of this work is to describe the boundary of the orbit of the $3\times
3$ determinant, that is, the set of limit points of the orbit that are not in the
orbit.

Let~$\detdrei$ be the polynomial
\[ \detdrei \eqdef \det
  \begin{pmatrix}
    x_1 & x_2 & x_3 \\
    x_4 & x_5 & x_6 \\
    x_7 & x_8 & x_9
  \end{pmatrix}  \in \bC[x_1,\dotsc,x_9],
  \]
  which we consider as a homogeneous form of degree~$3$ on the space~$\bC^{3\times 3}$ of~$3\times 3$ matrices, denoted~$W$.
Let~$\bC[W]_3$ denote the $165$-dimensional space of all homogeneous forms of degree~$3$ on~$W$.
The group~$\GL(W)$ acts on~$\bC[W]_3$ by right composition. For a nonzero~$P\in \bC[W]_3$,
let~$\Omega(P)$ denote the (projective) orbit of~$P$, namely the set of all~$[P \circ a] \in \bP(\bC[W]_3)$, with~$a\in \GL(W)$.
The \emph{boundary} of the orbit of~$P$, denoted~$\partial\Omega(P)$, is~$\overline{ \Omega(P) } \setminus \Omega(P)$,
where~$\overline{\Omega(P)}$, denoted also~$\overline\Omega(P)$, is the Zariski closure of the orbit in~$\bP(\bC[W]_3)$.

Our main result is a description of~$\partial\Omega(\detdrei)$ that answers a
question of \textcite[Problem~5.4]{La14}: The two known components are the only
ones. 
In \S\ref{sec:constr} we explain the construction of the two components.
Our contribution lies in~\S\ref{sec:only2} where we show that there is no other component.

\begin{theorem}
  The boundary~$\partial \Omega(\detdrei)$ has exactly two irreducible components:
  \begin{itemize}[leftmargin=4ex]
    \item The orbit closure of the determinant of the generic traceless matrix, namely
      \[ P_1 \eqdef \det \begin{pmatrix}
          x_1 & x_2 & x_3 \\
          x_4 & x_5 & x_6 \\
          x_7 & x_8 & -x_1-x_5
        \end{pmatrix}; \]
    \item The orbit closure of the universal homogeneous polynomial of degree two in three variables, namely
      \[ P_2 \eqdef x_4 \cdot x_1^2 + x_5\cdot x_2^2 + x_6\cdot x_3^2 + x_7\cdot x_1x_2 + x_8\cdot x_2x_3 + x_9\cdot x_1x_3. \]
  \end{itemize}
  \label{theo:main}
\end{theorem}

The two components are different in nature: the first one is the orbit
closure of a polynomial in only eight variables and is included in the orbit of $[\detdrei]$ under the action of~$\End W$; the second is more subtle and is not included in the $\End W$-orbit of~$[\detdrei]$.
Both components have analogues in higher dimension and some results are known about them~\parencite{landsberg_hypersurfaces_2013}.

\subsection*{Acknowledgment}
We are grateful to Peter Bürgisser for his many insightful comments about this work.

\section{Construction of two components of the boundary}
\label{sec:constr}

For~$P \in \bC[W]_3 \setminus \left\{ 0 \right\}$, let~$H(P) \subset \GL(W)$ denote its stabilizer,
that is
\[ H(P) \eqdef \left\{ a \in \GL(W) \st P\circ a = P \right\}. \]
The stabilizer~$H(\detdrei)$ is generated by the transposition map~$A\mapsto A^T$ and the
maps~$A \mapsto U A V$, with~$U$ and~$V$ in~$\SL(\bC^3)$ \autocite{Die49}.

\begin{lemma}  \label{lem:dimorbit}
  For any~$P \in \bC[W]_3$, $\dim \Omega(P) = 80 - \dim H(P)$,
  In particular, $\dim \Omega(\detdrei) = 64$ and $\dim \Omega(P_1) =\dim \Omega(P_2) = 63$.
\end{lemma}

\begin{proof}
  An easy application of the fiber dimension theorem to the
  map~$a \in \GL(W) \mapsto P\circ a \in \bC[W]_3$ gives that the dimension of the orbit of~$P$ in~$\bC[W]_3$
  is~$81 - \dim H(P)$.
  Since the projective orbit in~$\bP(\bC[W]_3)$ has one dimension less, the first claim follows.
  
  The stabilizer~$H(\detdrei)$ has dimension~$16$, hence~$\dim
  \Omega(\detdrei) = 64$.
  To compute the dimension of~$H(P_i)$, $1\leq i \leq 2$, one can compute
  the dimension of its Lie algebra defined as
  \[ T_1 H(P_i) = \left\{ a \in \End(W) \st P(x + t a(x)) = P(x) + O(t^2)  \right\}. \]
  It amounts to computing the nullspace of a $165\times 81$ matrix, which is easy using a computer.
\end{proof}

\begin{lemma}
  The boundary~$\partial\Omega(\detdrei)$ is pure of dimension~$63$.
  \label{lem:pure}
\end{lemma}

\begin{proof}
  Let~$\Omega'(\detdrei)$ be the affine orbit of~$\detdrei$ in~$\bC[W]_3$ under the action of~$\GL(W)$.
  It is isomorphic to~$\GL(W)/H(\detdrei)$, which is an affine variety because~$H(\detdrei)$ is reductive \autocite[\S4.2]{VinPop94}.
  Therefore~$\Omega'(\detdrei)$ is an affine open subset of its closure, it follows
  that the complement of~$\Omega'(\detdrei)$ in its closure is pure of codimension~$1$~\autocite[Corollaire 21.12.7]{EGA-IV-4},
  and the same holds true after projectivization.
\end{proof}

Let~$\varphi$ be the rational map
\begin{equation}
  \varphi : [a]\in \bP(\End W) \dashrightarrow [\detdrei \circ a] \in \overline\Omega(\detdrei).
  \label{eq:defphi}
\end{equation}
Let also~$Z$ be the irreducible hypersurface of~$\bP(\End W)$
\[ Z \eqdef \left\{ [a]\in \bP(\End W) \st \det(a) = 0 \right\}. \]
Note the difference between~$\detdrei\circ a$, which is a regular function of~$W$, and~$\det(a)$, which is a scalar.
The indeterminacy locus of~$\varphi$ is a strict subset of~$Z$.
By definition, $\Omega(\detdrei) = \varphi(\bP(\End W) \setminus Z)$.
Let~$\varphi(Z)$ denote the image of the set of the points of~$Z$ where~$\varphi$ is defined.

\begin{lemma}
  The closure~$\overline{\varphi(Z)}$ is an irreducible component of~$\partial \Omega(\detdrei)$.
  Furthermore~$\overline{\varphi(Z)} = \overline\Omega(P_1)$.
  \label{lem:orbit1}
\end{lemma}

\begin{proof}
  The closure~$\overline{\varphi(Z)}$ is clearly contained
  in~$\overline\Omega(\detdrei)$ since~$\GL(W)$ is dense in~$\End(W)$.  The
  image~$\varphi(Z)$ does not intersect~$\Omega(\detdrei)$: To show this, let us
  consider the function~$\nu : \bC[W]_3\to \bN$ which associates to~$P$ the
  dimension of the linear subspace of~$\bC[W]_2$ spanned by the partial
  derivatives~$\frac{\partial P}{\partial x_1}$,\,\dots\,, $\frac{\partial
  P}{\partial x_9}$. The function~$\nu$ is invariant under the action
  of~$\GL(W)$. 
  Because every form in~$\varphi(Z)$
  can be written as a polynomial in at most~$8$ linear forms, $\nu(P)\leq 8$ for all~$P\in\varphi(Z)$.
  On the other hand, $\nu(\det_3)=9$ and so~$\nu(P) = 9$ for any~$P\in\Omega(\detdrei)$.
  This shows that~$\varphi(Z)\cap\Omega(\detdrei) = \varnothing$.
  Thus~$\overline{\varphi(Z)}$ is contained in the boundary~$\partial\Omega(\detdrei)$.
  Moreover~$\overline{\varphi(Z)}$ is irreducible because~$Z$ is.
  
  Clearly~$P_1 \in \varphi(Z)$ and by Lemma~\ref{lem:dimorbit}, $\Omega(P_1)$ has dimension~$63$.
  Since
  \[ \overline\Omega(P_1) \subset \overline{\varphi(Z)} \subset \partial\Omega(\detdrei), \]
  they all three have dimension~$63$ and $\overline\Omega(P_1) = \overline{\varphi(Z)}$
  because the latter is irreducible. This gives a component of~$\partial \Omega(\detdrei)$.
\end{proof}

\begin{lemma}
  The orbit closure~$\overline\Omega(P_2)$ is an irreducible component of~$\partial \Omega(\detdrei)$ and is distinct from~$\overline\Omega(P_1)$.
  \label{lem:orbit2}
\end{lemma}

\begin{proof}
  We first prove that~$[P_2] \in \partial\Omega(\detdrei)$.
  Let
  \[ A = \begin{pmatrix} 0 & x_1 & -x_2 \\ -x_1 & 0 & x_3 \\ x_2 & - x_3 & 0.  \end{pmatrix}
    \text{ and }
    S = \begin{pmatrix} 2x_6 &  x_8 & x_9 \\ x_8 & 2x_5 & x_7 \\ x_9 & x_7 & 2x_4 \end{pmatrix}.
    \]
  By Jacobi's formula, $\det(A + tS) = \det A + \mop{Tr}( \mop{adj}(A) S ) t + o(t)$,
  where~$\mop{adj}(A)$ is
  the adjugate matrix of~$A$, which equals~$u^T u$ with~$u= (x_3, x_2, x_1)$.
  Since~$\det(A) = 0$, the projective class of the polynomial~$\det(A+tS)$
  tends to~$[\mop{Tr}( \mop{adj}(A) S )]$ when~$t\to 0$, and by construction, this limit is a point in~$\overline \Omega(\detdrei)$.
  Besides
  \[ \mop{Tr}( \mop{adj}(A) S ) = u S u^T = 2P_2, \] 
  thus~$[P_2] \in \overline\Omega(\detdrei)$.
  Yet~$[P_2]$ is not in~$\Omega(\detdrei)$, because its orbit has dimension~$63$, by Lemma~\ref{lem:dimorbit}, whereas the orbit of every point of~$\Omega(\detdrei)$
  is~$\Omega(\detdrei)$ itself.
  Therefore~$[P_2]$ is in the boundary~$\partial\Omega(\detdrei)$.
  Since~$\Omega(P_2)$ has dimension~$63$, this gives a compoment of~$\partial\Omega(\detdrei)$.
  It remains to show that~$[P_2]$ is not in~$\Omega(P_1)$, and indeed~$\nu(P_2) = 9$ whereas~$\nu(P_1) =8$, where~$\nu$ is the function introduced in the proof of Lemma~\ref{lem:orbit1}.
\end{proof}

Note that Lemma~\ref{lem:orbit2} generalizes to higher dimensions: the limit of the determinant on the space of skew-symmetric matrices always leads to a component of the boundary of the orbit of~$\det_n$, when~$n \geq 3$ is odd, as shown by \textcite[Prop.~3.5.1]{landsberg_hypersurfaces_2013}.

\section{There are only two components}
\label{sec:only2}

Let~$E$ denote~$\End(W)$ and recall the rational map~$\varphi : \bP(E) \dashrightarrow
\overline\Omega(\detdrei)$ defined in~\eqref{eq:defphi}. Let~$B \subset \bP(E)$
denote the indeterminacy locus of~$\varphi$, that is, the set of
all~$[a]\in\bP(E)$ whose image~$a(W) \subset W$ contains only singular
matrices. The locus $B$ is a subset of~$Z$ because every~$a$ not in~$Z$ is surjective and thus
has invertible matrices in its image.
One way to describe the orbit closure~$\overline\Omega(\detdrei)$ is
to give a resolution of the indeterminacies of the rational map~$\varphi$, that
is a, projective birational morphism~$\rho : X\to \bP(E)$ such that~$\varphi\circ \rho$
is a regular map. In this case, the regular map~$\varphi\circ\rho$ is projective
and therefore its image is closed and equals~$\overline\Omega(\detdrei)$.
As we will see, it is actually enough to resolve the indeterminacies of~$\varphi$ on some open subset of~$\bP(E)$.

Let~$H = H(\detdrei) \subset \GL(W)$ denote the stabilizer of~$\detdrei$
described above.  The group~$H$ acts on~$\bP(E)$ by left multiplication and the
rational map~$\varphi$ is~$H$-invariant: for~$a\in \End(W)$ and~$h\in H$,
$ \varphi( [ha] ) = [\detdrei \circ h \circ a] = \varphi([a])$. 
Let~$\bP(E)^\sst$ be the open subset of all semistable points in~$\bP(E)$ under the action of~$H$, that is the
set of all~$[a]\in \bP(E)$ such that there exists a non-constant homogeneous
$H$-invariant regular function~$f\in \bC[E]^H$ on~$E$ such that~$f(a)\neq
0$.  Equivalently \autocite[\S4.6]{VinPop94}, the complement of~$\bP(E)^\sst$ is the set of
all~$[a] \in \bP(E)$ such that~$0$ is in the closure of~$Ha$ in~$E$.
Let~$X$ be the closure in~$\bP(E)^\sst \times \overline\Omega(\detdrei)$ of the
graph of the rational map~$\varphi$, namely
\[ X \eqdef \mop{Closure}\left\{ \left( [a], [P] \right) \in \bP(E)^\sst \times \overline\Omega(\detdrei) \st [P] = [\detdrei \circ a] \right\}. \]
Let~$\rho: X\to \bP(E)^\sst$ denote the first projection.
By construction, it is the blowup of~$\bP(E)^\sst$ along the ideal sheaf defined by the
condition~$\detdrei \circ a = 0$, whose support is the indeterminacy
locus~$B\cap \bP(E)^\sst$. (The condition~$\detdrei \circ a = 0$ expands into 165~homogeneous polynomials of degree~3
in the 81~coordinates of~$a$.)

The variety~$X$ also carries a regular map~$\psi :
X\to\overline\Omega(\detdrei)$ given by the second projection. By construction,
it resolves the indeterminacies of~$\varphi$ on~$\bP(E)^\sst$: the
rational map~$\varphi\circ\rho : X\to \overline\Omega(\detdrei)$ extends to a regular map which equals~$\psi$.

\begin{lemma}
  $\psi(X) = \overline \Omega(\detdrei)$.
  \label{lem:close-image}
\end{lemma}

\begin{proof}
  The image of~$\varphi$, which is~$\Omega(\detdrei)$, is included in~$\psi(X)$
  and~$\psi(X) \subset \overline\Omega(\detdrei)$.
  Thus, it is enough to show that~$\psi(X)$ is closed.

  Let~$T$ be the projective variety~$\bP(E)\times \bP(\bC[W]_3)$. The group~$H$
  acts on~$T$ by~$h\cdot (a,P) = (h\cdot a, P)$.  Let~$T^\sst$ the open subset
  of semi-stable points for this action; clearly~$T^\sst = \bP(E)^\sst\times
  \bP(\bC[W]_3)$.  The GIT quotient~$T^\sst\dquo H$ is a projective variety and
  the canonical morphism~$\pi : T^\sst \to T^\sst\dquo H$ maps~$H$-invariant
  closed subsets to closed subsets \autocite[e.g.][\S4.6]{VinPop94}, in
  particular~$\pi(X)$ is closed.  Moreover, the map~$\psi$ is~$H$-invariant so
  it factors as~$\psi'\circ\pi$ for some regular map~$\psi': T^\sst \dquo H
  \to \bP(\bC[W]_3)$.  The image~$\pi(X)$ is closed in the projective variety~$T^\sst\dquo H$
  thus~$\psi'(\pi(X))$ is closed.  This proves the claim since the latter is
  just~$\psi(X)$.
\end{proof}

The construction of~$X$ follows a general method to resolve the indeterminacies
of a rational map, and as such, it gives little information.  In fact~$X$ is a
blowup of~$\bP(E)^\sst$ along a smooth variety.

First of all, the indeterminacy locus~$B$ is precisely known, thanks to the classification of the maximal linear subspaces of~$E$
containing only singular matrices \autocites{Atk83}{FLR85}{EisHar88}.
Let~$H^0$ denote the connected component of~$1$ in~$H$ ---~due to the
transposition map, $H$ has two components.
For every~$[a]\in B$, there is a~$h \in
H^0$ such that~$(ha)(W)$ is a subset of one of the following spaces of
singular matrices:
\[
\begin{pmatrix}
* & * & * \\
* & * & * \\
0 & 0 & 0
\end{pmatrix},\ 
\begin{pmatrix}
* & * & 0 \\
* & * & 0 \\
* & * & 0 
\end{pmatrix},\ 
\begin{pmatrix}
* & * & * \\
* & 0 & 0 \\
* & 0 & 0
\end{pmatrix} \text{ and } 
\begin{pmatrix}
0 & \alpha & -\beta \\
-\alpha & 0 & \gamma \\
\beta & -\gamma & 0
\end{pmatrix},\ \alpha,\beta,\gamma\in\bC.
\]
The first three are called \emph{compression spaces}, and the fourth is the space of~$3\times3$ skew-symmetric matrices, denoted~$\Lambda_3$.
They give four  components of~$B$. Let~$B_1$, $B_2$, $B_3$ and~$\Bsk$ denote them, respectively.
For example
\[ \Bsk = \left\{ [a]\in\bP(E) \st \exists U,V\in \SL(\bC^3) : \forall p\in W : U a(p) V \in \Lambda_3 \right\}. \]

\begin{lemma}
  We have~$B \cap \bP(E)^\sst = \Bsk \cap \bP(E)^\sst \neq \varnothing$.
  \label{lem:sscomp}
\end{lemma}

\begin{proof}
  It is easy to check that the three matrices
  \[
    \left(\begin{smallmatrix}
      t & 0 & 0 \\
      0 & t & 0 \\
      0 & 0 & t^{-2}
    \end{smallmatrix}\right)
    \left(\begin{smallmatrix}
      * & * & * \\
      * & * & * \\
      0 & 0 & 0
    \end{smallmatrix}\right),\ 
    \left(\begin{smallmatrix}
      * & * & 0 \\
      * & * & 0 \\
      * & * & 0
    \end{smallmatrix}\right)
    \left(\begin{smallmatrix}
      t & 0 & 0 \\
      0 & t & 0 \\
      0 & 0 & t^{-2}
    \end{smallmatrix}\right),\ 
    \left(\begin{smallmatrix}
      t^2 & 0 & 0 \\
      0 & t^{-1} & 0 \\
      0 & 0 & t^{-1}
    \end{smallmatrix}\right)
    \left(\begin{smallmatrix}
      * & * & * \\
      * & 0 & 0 \\
      * & 0 & 0
     \end{smallmatrix}\right)
     \left(\begin{smallmatrix}
      t^2 & 0 & 0 \\
      0 & t^{-1} & 0 \\
      0 & 0 & t^{-1}
    \end{smallmatrix}\right) 
  \]
all tend to~$0$ when~$t\to 0$, for any constants~$*$.
This proves that~$B_1$, $B_2$ and~$B_3$ do not meet~$\bP(E)^\sst$.

To show that~$B\cap\bP(E)^\sst$ is not empty, pick any three points~$p_1$, $p_2$ and~$p_3$ in~$W$.
The function
\[ \tau : a \in E \mapsto \mop{Tr}\big( a(p_1) \cdot \mop{adj}( a(p_2) ) \cdot a(p_3) \cdot \mop{adj}( a(p_1 +p_2 +p_3) ) \big) \in \bC, \]
is~$H^0$-invariant: if~$h\in H$ is the map~$A\mapsto UAV$, for some~$U,V\in \SL(\bC^3)$,
then
\begin{align*}
  \tau(h a) &= \mop{Tr}\big( U a(p_1) V \cdot V^{-1} \mop{adj}( a(p_2) ) U^{-1} \cdot Ua(p_3)V \cdot V^{-1}\mop{adj}( a(p_1 +p_2 +p_3)) U^{-1} \big),
\end{align*}
which equals~$\tau(a)$.
It follows that the function~$a\mapsto \tau(a)+\tau(T a)$ is~$H$-invariant, where~$T : A \mapsto A^T$ is the transposition map.
Consider the
function~$b : W \to W$
defined by
\begin{equation}
b = \left(\begin{smallmatrix}
0 & x_1 & -x_2 \\
-x_1 & 0 & x_3 \\
x_2 & -x_3 & 0
\end{smallmatrix}\right),\
  \label{eq:defb}
\end{equation}
where the~$x_i$'s are linear forms~$W\to \bC$.
This gives a point~$[b]$ in~$\Bsk$ 
If the points~$p_i$'s are generic, then a simple computation shows that~$\tau(b)+\tau(Tb)\neq 0$.
\end{proof}

\begin{lemma}
  The subvariety~$\Bsk\cap \bP(E)^\sst$ is smooth and~$\rho : X\to \bP(E)^\sst$ is the blowup of~$\bP(E)^\sst$ along it.
  \label{lem:reduced}
\end{lemma}

\begin{proof}
  Let~$\cI$ be the ideal sheaf generated by the condition~$\detdrei \circ a =
  0$.  its support is clearly~$B\cap \bP(E)^\sst$, which is also~$\Bsk\cap
  \bP(E)^\sst$, by Lemma~\ref{lem:sscomp}.  By definition, $X$ is the blowup
  of~$\bP(E)^\sst$ along~$\cI$.  By contrast, the blowup of~$\bP(E)^\sst$
  along~$\Bsk\cap \bP(E)^\sst$ is defined to be the blowup of the \emph{reduced}
  ideal sheaf whose support is~$\Bsk\cap \bP(E)^\sst$.
  Thus, it is enough to check that~$\cI$ is smooth (which implies reduced).  Let~$[b]\in \Bsk$ be the point
  defined in~\eqref{eq:defb}.
  
  We first observe that~$\Bsk\cap \bP(E)^\sst = [H \cdot b \cdot \GL(W)]$, the orbit
  of~$[b]$ under the left action of~$H$ and the right action of~$\GL(W)$ by multiplication.  The
  right-to-left inclusion is clear because the left-hand side is invariant
  under both actions and contains~$[b]$.  Conversely, let~$[a]\in \Bsk\cap \bP(E)^\sst$. By
  definition of~$\Bsk$, we may assume that the image of~$a$ is included in~$\Lambda_3$, up to replacing~$a$ by another point in its orbit~$Ha$.  If the
  image of~$a$ had dimension~$2$ or less, then~$a$ would also lie in some of
  the~$B_i$'s, $1\leq i\leq 3$ \autocite{BurDra06}.\footnote{\textcite[Theorem~2 and the discussion above it]{BurDra06}, states that a subspace of~$E$ of dimension~$2$ containing only singular matrices is contained in a compression space.}
  Since~$[a]\in \bP(E)^\sst$, Lemma~\ref{lem:sscomp} ensures that~$a$ is not in
  one of the~$B_i$'s, thus~$a$ has rank~$3$ and its image is~$\Lambda_3$. Then
  there is a~$g\in\GL(W)$ such that~$a= bg$, and thus~$a \in H \cdot b \cdot
  \GL(W)$.
  

  Regarding the smoothness,
  since~$\cI$ is invariant under the action
  of~$H$ and~$\GL(W)$ and since~$\Bsk\cap \bP(E)^\sst$ is an orbit under the same action,
  it is enough to check that~$\cI$ 
  is smooth at one point, say~$[b]$.
  By the Jacobian criterion \autocite[\S V.3]{EH00}, it is enough to check that the dimension of the tangent space
  \[ T = \left\{ c \in T_{[b]}\bP(E) \st \forall p\in W, \det(b(p) + t c(p)) = O(t^2) \right\}, \]
  equals the dimension of~$\Bsk$ at~$[b]$.
  The dimension of~$T$ is easily computed using a computer: it is equal to~34.
  To compute the dimension of~$\Bsk$, we use again the fact that it is an orbit under a group action: it is smooth and the tangent space at~$[b]$ equals
  \begin{align*}
    T_{[b]} \Bsk &= \left\{ m b + b c \st m\in T_1H,\ c\in T_1 \GL(W) \right\} \subset T_{[b]}\bP(E)\\
    &= \left\{ p \in W \mapsto M b(p) + b(p) N + b(c(p)) \in W \st M,N \in W,\ c\in\End(W) \right\}.
  \end{align*}
  Using a computer, we find that this space has also dimension~$34$, which terminates the proof.
\end{proof}

\begin{proof}[Proof of Theorem~\ref{theo:main}]
  Let~$D$ be the inverse image of the hypersurface~$Z$ by the blowup~$\rho$.
  $D$ is a hypersurface with exactly two irreducible components
  because~$\bP(E)$ is smooth and because the center of the blowup~$\rho$ is
  also smooth and included in~$Z$ \autocite[Lecture~7]{Ha95}. Respectively, the two
  components are the exceptional divisor~$\rho^{-1}(\Bsk)$ and the
  strict transform of~$Z$, i.e. the closure of~$\rho^{-1}(Z \setminus
  \Bsk)$.

  On the other hand~$\psi(X\setminus D) = \varphi(\GL(W)) = \Omega(\detdrei)$,
  thus~$\partial\Omega(\detdrei) \subset \psi(D)$, by Lemma~\ref{lem:close-image}.
  This proves that~$\partial\Omega(\detdrei)$ has at most two components:
  The components found in~\S\ref{sec:constr} are the only ones.\footnote{Though it is not necessary,
    we check easily that the image of the exceptional divisior
  is~$\overline\Omega(P_2)$ while the image of the strict transform of~$Z$
  gives~$\overline\Omega(P_1)$.
} This finishes the proof of Theorem~\ref{theo:main}.
\end{proof}


\raggedright
\printbibliography

\end{document}